\documentclass[11pt]{amsart}
\usepackage{verbatim}

\newtheorem{theorem}{Theorem}[section]
\newtheorem{proposition}[theorem]{Proposition}

\newtheorem{lemma}[theorem]{Lemma}

\newtheorem{cor}[theorem]{Corollary}

\theoremstyle{plain}
\numberwithin{equation}{theorem}

\theoremstyle{remark}
\newtheorem{remark}[theorem]{Remark}

\newcommand{\bN}{\mathbb{N}}
\newcommand{\bA}{\mathbb{A}}

\newcommand{\bC}{{\mathbb C}}
\newcommand{\bZ}{{\mathbb Z}}
\newcommand{\bP}{{\mathbb P}}

\newcommand{\fo}{{\mathfrak o}}
\newcommand{\bQ}{{\mathbb Q}}

\newcommand{\lra}{\longrightarrow}

\newcommand{\cU}{\mathcal U}

\DeclareMathOperator{\Pic}{Pic}
\DeclareMathOperator{\PGL}{PGL}

\begin{document}
\title{Integral points in two-parameter orbits}
\author[Corvaja, Sookdeo, Tucker, Zannier]{Pietro Corvaja, 
Vijay Sookdeo, Thomas J. Tucker, Umberto Zannier}

\address{Pietro Corvaja,
Dipartimento di Matematica e Informatica,
Universit\`a di Udine, \\
Via delle Scienze, 206 \\
33100 Udine\\
Italy
}
\email{pietro.corvaja@uniud.it}

\address{
Vijay Sookdeo,
Department of Mathematics \\
The Catholic University of America \\
Cardinal Station \\
Washington, DC 20064
}
\email{sookdeo@cua.edu}

\address{Thomas J. Tucker,
Department of Mathematics\\
Hylan Building\\
University of Rochester\\
Rochester, NY 14627
}

\email{ttucker@math.rochester.edu}

\address{Umberto Zannier,
Scuola Normale Superiore,\\
Piazza dei Cavalieri, 7\\
56126 Pisa\\
 Italy
}

\email{u.zannier@sns.it}

\maketitle

\begin{abstract}
  Let $K$ be a number field, let $f: \bP_1 \lra \bP_1$ be a
  nonconstant rational map of degree greater than 1, let $S$ be a finite set of places of $K$,
  and suppose that $u, w \in \bP_1(K)$ are not preperiodic under $f$.
  We prove that the set of $(m,n) \in \bN^2$ such that $f^{\circ
    m}(u)$ is $S$-integral relative to $f^{\circ n}(w)$ is finite and
  effectively computable.  This may be thought of as a two-parameter
  analog of a result of Silverman on integral points in orbits of
  rational maps.  

This issue can be translated in terms of integral points on an open subset  of $\bP_1^2$; then one can apply a modern version of the method of Runge, after increasing the number of components at infinity by iterating the rational map. Alternatively, an ineffective result comes from a well-known theorem of Vojta.
\end{abstract}

\section{Introduction}\label{intro}

In 1929, Siegel \cite{Siegel} proved that if $C$ is an irreducible
affine curve defined over a number field $K$ and $C$ has at least
three points at infinity, then there are at most finitely many
$K$-rational points on $C$ that have integral coordinates. When $C$
has positive genus, something stronger is true: any affine
curve has defined over a number field has at most finitely many
$K$-rational points that have integral coordinates.  Silverman
\cite[Theorem A]{SilSiegel} later gave a dynamical variant of Siegel's
theorem, proving that if $f: \bP_1 \lra \bP_1$ is a rational
function such that $f^{\circ 2}$ is not a polynomial and $u \in K$
is not preperiodic for $f$, there are only finitely many $n$ such
that $f^{\circ n}(u)$ is integral relative to the point at
infinity (we will give a  full definition of what it means to be integral
relative to a point in Section~\ref{nota}).  Moreover, \cite[Theorem A]{SilSiegel} can be made
effective, as we shall see
Section 5.

Recently, various authors (see \cite{Denis, Bell, JNT}) have proposed
a dynamical analog of the Mordell-Lang conjecture for semiabelian
varieties. The Mordell-Lang conjecture for semiabelian varieties,
which was proved by Faltings \cite{Faltings} and Vojta \cite{PV2},
states that if a finitely generated subgroup $\Gamma$ of a semiabelian
variety $A$ over $\bC$ intersects a subvariety $V \subseteq S$ in
infinitely many points, then $V$ must contain a positive-dimensional
algebraic subgroup of $A$.  One dynamical analog asserts that if one
has a morphism of varieties $\Phi: X \lra X$ defined over
$\bC$, a subvariety $V \subseteq X$, and a point $\alpha$ in
$X(\bC)$, then the forward orbit of $\alpha$ under $\Phi$ (that is,
the set of distinct iterates $\Phi^{\circ n}(\alpha)$) may intersect
$V$ infinitely often only if $V$ contains a $\Phi$-periodic subvariety
of $X$ (that is, a subvariety of $W$ of $X$ such that $\Phi^{\circ
  n}(W) = W$ for some $n > 0$) having positive dimension.  Note, however, that since the forward
orbit of a point under a single map is parametrized by the positive
integers, it is more analogous to a cyclic group $\Gamma$ than it is
to an arbitrary finitely generated group $\Gamma$.  Thus, one might
ask for a ``multi-parameter'' dynamical conjecture concerning the
forward orbit of $\alpha$ under a finitely generated semigroup of
commuting maps.  In \cite{GTZ2}, this problem is considered and
results are obtained in the case where the subvarieties $V$ are lines
in $\bA^2$ and the semigroup of morphisms is the set of all $(f^{\circ
  m}, g^{\circ n})$ for fixed polynomials $f$ and $g$.
 
The dynamical variants of the Mordell-Lang conjecture described above
all pose questions about the intersection of forward orbits with
subvarieties.  Here we consider the related problem of integral points
in forward orbits. The main theorem of this paper is the following,
which may be thought of as a two-parameter version of \cite[Theorem
A]{SilSiegel}.
 
\begin{theorem}\label{main}
Let $K$ be a number field and $S$ a finite set of primes in $K$.  Let
$f: \bP_1 \lra \bP_1$ be a rational function with degree $d \geq 2$
that is not conjugate to a powering map $x^{\pm d}$, and let $u,
w \in \bP_1(K)$ be points that are not preperiodic for
$f$. Then the set of $(m,n) \in \bN^2$ such that $f^{\circ m}(u)$
is $S$-integral relative to $f^{\circ n}(w)$  is finite and effectively
computable.
\end{theorem}

Clearly the set $(m,n)$ such that $f^{\circ m}(u)$
is $S$-integral relative to $f^{\circ n}(w)$ depends on $u$ and $w$.
It is possible, however, to prove an effective degeneracy result for integral
points depending only on $f$, $S$, and $K$.  This is stated in
Theorem~\ref{eff-degen}, which is phrased in terms of the
$S$-integrality of points $(f^{\circ m}(u), f^{\circ n}(w))$ relative
to inverse images of the diagonal in $\bP_1^2$.   
\bigskip

The outline of the paper is as follows.    
In Section \ref{nota}, after introducing some notation, we give some equivalent notions of integrality. This will reduce  our problem to the study of integral points on the complement in $\bP_1^2$ of suitable divisors.
Then, we show that a noneffective version of Theorem
\ref{main} can be obtained very quickly by combining
\cite[Appendix]{BGKTZ} with \cite[Theorem 2.4.1]{PV}; this is
Theorem~\ref{ineff-fin}.  

In Section \ref{dege}, we prove
Theorem~\ref{eff-degen}, an effective degeneracy result for
$K$-rational in $\bP_1^2$ that are $S$-integral relative to inverse
images of the diagonal under $(f,f)$.  

The technique here  
originates from  Runge's theorem \cite{Runge}; Runge treated only the case of curves, 
but see  \cite[Section 9.6]{BG} and
\cite{LevinRunge} for a modern account and higher dimensional generalizations. 
Since inverse images of the
diagonal under $(f,f)$ have several components, one can construct many rational
functions $\psi$ whose pole divisors are supported on these inverse images.  The
main difficulty is dealing with the points at which many
components of these pole divisors meet; this can be overcome by
introducing rational functions
that vanish to a high degree at these intersection points.  Then,  in
Section \ref{fini}, we use Theorem~\ref{eff-degen} and some simple
facts about periodic curves to finish the proof of
Theorem~\ref{main}.  We end with a few remarks about what happens when
$f$ is conjugate to a powering map or $u$ or $w$ is preperiodic.   

\section{Notation}\label{nota}

Let $\bP_n$ denote the projective $n$-space and write $\bP_n^m$ for
$m$-fold Cartesian product of $\bP_n$.  We fix projective coordinates
$[z_1: \dots : z_n]$.  When convenient, we regard $\bP_1(K)$ as
$K\cup \{\infty\}$ and work in affine coordinates.  

Let $P=[x_0:\dots:x_n]$ and $Q=[y_0:\dots:y_n]$ be points in $\bP_n$, and
define 
$$ 
\|P \|_v = \left\{ \begin{array}{ll}
 \max(|x_0|_v,\dots, |x_n|_v) &\mbox{ if $v$ is non-archimedean} \\ \\
 \sqrt{|x_0|_v^2+\dots+|x_n|_v^2} &\mbox{ if $v$ is archimedean}.
       \end{array} \right.
$$  
Then the \emph{Chordal Distance} on $\bP_n$ is given by
$$\Delta_v(P,Q) = \frac{\max_{0\le i\le j \le n}\{ |x_iy_j - x_jy_i|_v \}
}{\|P\|_v \|Q\|_v}.$$
This definition is independent of the choice of projective coordinates for $P$
and $Q$, and satisfies $0 \le \Delta_v(P,Q)\le 1$.  Additionally,  $\Delta_v$
changes by a uniformly bounded factor under an automorphism of $\bP_n$. 
Viewing $\bP_1^2$ as a projective subvariety of $\bP_3$ via the Segre embedding,
we see that $P=(x,y) \in \bP_1^2$ satisfies $\Delta_v(P,(\infty, \infty)) <
\delta$ if and only if $|x|_v > 1/\delta$ and $|y|_v>1/\delta$.

Let $f(x)=p(x)/q(x)$, with $p(x), q(x)$ coprime,  be a rational function of
degree $d$, and write $f^{\circ n}(x)=p_n(x)/q_n(x)$ where $p_n(x)$ and
$q_n(x)$ are also coprime.  The homogenization of $f^{\circ n}$ gives the
rational function $F^{\circ n}([x_0:x_1])=[P_n(x_0,x_1):Q_n(x_0,x_1)]$ on
$\bP_1$.  Let
$D_n$ be divisor of zeros for $P_n(x_0,x_1) Q_n(y_0,y_1) - P_n(y_0,y_1) Q_n
(x_0,x_1)$ and $B_i := D_i - D_{i-1}$. We note that $B_i$ is an effective divisor. In fact,
put  $x = x_0/x_1$ and $y = y_0/y_1$ and  let  
\[\beta_i = \frac{f^{\circ i}(x) - f^{\circ i}(y)}{f^{\circ (i-1)}(x) - f^{\circ
(i-1)}(y)}.\]
Then $B_i$ is the zero divisor of $\beta_i$, in particular it is effective.  

Having fixed coordinates on $\bP_1$, we have models
$(\bP_1)_{\fo_K}$ and $(\bP_1^2)_{\fo_K}$ for $\bP_1$
and $\bP_1^2$ over the ring of integers $\fo_K$  for a number field
$K$.  Then, for any finite set of places of $K$ including all the
archimedean places of $K$, we may define $S$-integrality in the usual
ways.  We say that a point $Q \in \bP_1(K)$ is $S$-integral with respect to
a point $P \in \bP_1(K)$ if the Zariski closures of $Q$ and $P$ in
$(\bP_1){\fo_K}$ do not
meet over any primes $v \notin S$; similarly, we say that a point $Q
\in   \bP_1^2(K)$ is $S$-integral with respect to a subvariety $V$ of
$\bP_1^2$, defined over $K$, if  if the Zariski closures of $Q$ and $V$ in
$(\bP_1^2){\fo_K}$ do not meet over any primes $v \notin S$.

Note that the diagonal $D_0$ in $\bP_1^2$ is defined by the equation $x_0
y_1 - y_0 x_1 = 0$. So, more concretely, we say that $([a_0:a_1],
[b_0:b_1])$ is $S$-integral relative to the diagonal $D_0$ if
\begin{equation}\label{weil}
 |a_0 b_1 - a_1 b_0|_v \geq \max( |a_0|_v, |a_1|_v) \cdot \max(
 |b_0|_v, |b_1|_v)
\end{equation}
for all $v \notin S$.  We say $[a_0:a_1] \in \bP_1(K)$ is $S$-integral relative
to $[b_0:b_1]
\in \bP_1(K)$ if $([a_0:a_1], [b_0:b_1])$ is $S$-integral relative to
$D_0$.  This definition of integrality is consistent with the previous one
given that involved models for $\bP_1$ and $\bP_1^2$.  
Hence, given two points $P,Q\in\bP_1(K)$,  

{\it $P$ is integral with respect to $Q$ if and only if the pair $(P,Q)\in\bP_1^2(K)$ is integral with respect to the diagonal.}
\smallskip

We may suppose, after enlarging our set of places $S$,  that our rational
function $f$ has good reduction at all
primes outside of $S$, that is that $p$ and $q$ have no common root at
any place outside of $S$.   Then we see that $([a_0:a_1],
[b_0:b_1])$ is $S$-integral relative to $D_n$ if 
\begin{equation*}
\begin{split} 
| P_n(a_0,a_1) & Q_n(b_0,b_1) -   P_n(b_0, b_1) Q_n(a_0,a_1)|_v \\
&  \geq \max( |a_0|_v, |a_1|_v)^{d^n} \cdot \max(
 |b_0|_v, |b_1|_v)^{d^n} 
\end{split}
\end{equation*}
for all $v \notin S$.  Note that from the definition above, it is clear that if a point is
$S$-integral relative to $D_n$ then it is also $S$-integral
relative to $D_m$ for any $m \leq n$ (as one would expect given that
the support of $D_m$ is contained in the support of $D_n$ when $m \leq
n$).  

Furthermore, if $S$ contains all the places of bad reduction for $f$, we have
\begin{equation}\label{functoriality}
\begin{split}
& \text{$([a_0:a_1], [b_0:b_1])$ is $S$-integral relative to $D_n$} \\
 & \quad \quad  \quad \text{ $\Longleftrightarrow$ $(f^{\circ n}([a_0:a_1]),
f^{\circ n}([b_0:b_1]))$ is $S$-integral relative to $D_0$. } 
\end{split}
\end{equation}

We will often use coordinates $(x,y)$ on $\bP_1^2$ where
$x = x_0/x_1$ and $y = y_0/y_1$ for projective coordinates
$([x_0:x_1],[y_0:y_1])$.  We write $(\infty, \infty)$ for the point
$([1:0],[1:0])$. \medskip

We say that point $z$ is {\it exceptional} for $f$ if $z$ is a totally
ramified fixed point of $f^{\circ 2}$.

\section{Ineffective finiteness}\label{ineff}

Applying a result of Vojta, we have the following.  

\begin{theorem}\label{ineff-degen}
  Let $K$ be a number field and $S$ a finite set of primes in $K$.
  Let $f: \bP_1 \lra \bP_1$ be a rational function of degree $d \geq
  2$. Then the set of points in $\bP_1^2(K)$ that are $S$-integral
  relative to $D_4$ lies in a proper closed subvariety $Z$ of $\bP_1^2$.
\end{theorem}
\begin{proof}
  The divisor $D_4$ has at least five irreducible components, since it
  contains $B_0, \dots, B_4$.  A theorem of Vojta \cite[Theorem
  2.4.1]{PV} asserts that for any divisor $W$ on a nonsingular variety
  $V$, the points in $V(K)$ that are $S$-integral points relative to $W$ are not
dense in $V$
  if $W$ has at least $\rho + r + \dim V + 1$ components, where
  $\rho$ and $r$ are the ranks of $\Pic^0(V)$ and the N\'eron-Severi
  group of $V$, respectively.  Since $\Pic^0(\bP_1^2)$ is trivial and
  $\bP_1^2$ has a N\'eron-Severi group of rank 2 (see \cite[Example
  6.6.1]{H}), it follows that the set of points in $\bP_1^2(K)$ that
  are $S$-integral relative to $D_4$ lies in a proper closed
  subvariety of $\bP_1^2$.
\end{proof}

\begin{cor}\label{strong1}
  Let $K$ be a number field and $S$ a finite set of primes in $K$.
  Let $f: \bP_1 \lra \bP_1$ be a rational function with degree $d \geq
  2$.  There is a proper closed subvariety of $Y$ of $\bP_1^2$ such
  that for any $u, w \in \bP_1(K)$, the subvariety $Y$ contains all but at most
  finitely many points
  $(f^{\circ m}(u), f^{\circ n}(w))$ for which $f^{\circ m}(u)$ is
  $S$-integral relative to $f^{\circ n}(w)$.
 \end{cor}
\begin{proof}
  Let $Z$ be as in Theorem~\ref{ineff-degen}, let $z_1, \dots, z_e$
  be the set of preperiodic points of $f$ in $K$ (note that this set
  must be finite), and let 
\[ Y = Z \cup \left( \bigcup_{i=1}^e \bP_1 \times \{ z_i \}  \right)
\cup \left( \bigcup_{i=1}^e \{ z_i \} \times \bP_1  \right). \]
If $u$ or $w$ is preperiodic for $f$, then $(f^{\circ m}(u), f^{\circ
    n}(w)) \in Y$ for all $m,n$, so we may assume that neither $u$ nor $w$ is
  preperiodic.  

  Now, by \cite[Theorem A]{SilSiegel}, for any fixed $n$, there are at most
  finitely many $m$ such that $f^{\circ m}(u)$ is $S$-integral
  relative to $f^{\circ n}(w)$, because no $f^{\circ n}(w)$ is
  exceptional; likewise, there are at most finitely many $n$ such
  that$f^{\circ n}(w)$ is $S$-integral relative to $f^{\circ m}(u)$.
  Thus, there are at most finitely many $(m,n)$ with $\min(m,n) \leq 4$ such that
  $(f^{\circ m}(u), f^{\circ n}(w))$ is $S$-integral relative to
  $D_0$.  By \eqref{functoriality}, we see that if $m,n \geq 4$, then
  $(f^{\circ (m-4)}(u), f^{\circ (n-4)}(w))$ is $S$-integral relative
  to $D_4$ if and only if $(f^{\circ m}(u), f^{\circ n}(w)$ is
  $S$-integral relative to $D_0$.  Applying Theorem~\ref{ineff-degen},
  one sees then that the set of points of the form $(f^{\circ m}(w),
  f^{\circ n}(u))$ that are $S$-integral relative to $D_0$ is
  contained in $Y$.  
\end{proof}

\begin{cor}\label{ineff-fin}
  Let $K$ be a number field and $S$ a finite set of primes in $K$.
  Let $f: \bP_1 \lra \bP_1$ be a rational function with degree $d \geq
  2$ that is not conjugate to a powering map $x^{\pm d}$, and let
  $u, w \in \bP_1(K)$ be points that are not preperiodic for
  $f$. Then the set of $(m,n) \in \bN^2$ such that $f^{\circ
    m}(u)$ is $S$-integral relative to $f^{\circ n}(w)$ is
  finite.
\end{cor}
\begin{proof}
Let $Y$ be as in Theorem~\ref{strong1}.
 The set of points in $Y$ that are $S$-integral
relative to $D_0$ is finite by the main theorem of the Appendix of
\cite{BGKTZ}, since $f$ is not conjugate to a powering map $x^{\pm d}$.
\end{proof}


\section{Effective degeneracy} \label{dege}
We will now prove the following theorem.

\begin{theorem}\label{eff-degen}
  Let $K$ be a number field and $S$ a finite set of places of $K$
  including all the archimedean places.  Let
  $f: \bP_1 \lra \bP_1$ be a rational function of degree $d \geq 2$ that
  is not conjugate to a powering map $x^{\pm d}$. Then there is a
  computable integer $k_0$ such that the set of points in $\bP_1^2(K)$
  that are $S$-integral relative to $D_{k_0}$ lies
  in an effectively computable proper closed subvariety of $\bP_1^2$.
\end{theorem}

It suffices to show that there is a computable constant $A$ and a computable set
$\Psi$ of nonzero rational functions $\psi$ on $\bP_1^2$ such that all the
points in $\bP_1^2(K)$ that are $S$-integral relative to $D_{k_0}$ lie on a
curve of the form $\psi(x,y) = \gamma$ with $h(\gamma) < A$ and $\gamma \in K$
(note that there are finitely many such $\gamma$ and they can be effectively
determined).  The remainder of this section is devoted to constructing such a
set $\Psi$ of rational functions.  

Much of the proof is devoted to exploring the behavior of the functions $\psi$
near points that lie on the intersections of their pole divisors.
We begin with a lemma about these intersection points.

\begin{lemma}\label{multiple}
If the intersection of $2d$ distinct divisors $B_{m_0},\dots,
B_{m_{2d -1}}$ contains a point
$p=(\xi,\eta)$, then there are distinct $i$ and $j$ such that $f^{\circ
 m_ i - 1}(\xi) = f^{\circ m_i -1}(\eta)$ is a periodic critical point with
period dividing some $m_j - m_i$. 
\end{lemma}

\begin{proof}
First, we note that if $(c,c) \in B_1 \cap B_0$, then $(c,c)$ has multiplicity
greater than 1 on $D_1$.  Since $D_1$ is defined by the equation $f(x)
= f(y)$, this means that $c$ must be a ramification point of $f$.  Now
suppose that $(\xi, \eta)$ are in $B_m$ and $B_n$ for $m
< n$.  Then $(f^{\circ (n-1)}(\xi), f^{\circ (n-1)}(\eta)) \in B_1
\cap B_0$ so $f^{\circ (n-1)}(\xi) = f^{\circ (n-1)}(\eta) = c$ for
$c$ a ramification point of $f$.  Thus, if  $(\xi,\eta) \in B_{m_0}
\cap \dots \cap B_{m_{2d -1}}$ for $m_0 < m_1 < \dots < m_{2d-1}$,
then $f^{\circ (m_k - 1)}(\xi) = f^{\circ (m_k - 1)}(\eta)$ is a
ramification point of $f$ for $k=1, \dots, 2d -1$.  Since $f$ has at
most $2d - 2$ ramification points, we must have $f^{\circ (m_i -
  1)}(\xi) = f^{\circ (m_j - 1)}(\xi)$ for some $i \not= j$ with
$i,j\in\{1,\ldots,2d-1\}$. 
\end{proof}

Since $f$ has only finitely many ramification points, we may choose an $M$
such that the period of each periodic ramification point divides
$M$. Note that a point $x$ is periodic for $f$ if and only if it is
periodic for $f^{\circ M}$, since if $(f^{\circ M})^{\circ k}(x) = x$,
then   $f^{\circ Mk}(x) = x$ and if $f^{\circ k}(x) = x$, then
$(f^{\circ M})^{\circ k}(x) = (f^{\circ k})^{\circ M}(x) = x$.  Thus,
every  periodic ramification point of $f^{\circ M}$ is a periodic ramification point for $f$ and is a fixed point of $f^{\circ M}$.  
Proving Theorem~\ref{eff-degen}  for an iterate $f^{\circ M}$ of $f$
is equivalent to proving it for $f$ itself, so we may suppose, in view of the previous remark, that all the periodic ramification points of $f$ are fixed points for $f$.  Since $f$ is
not conjugate to a powering map and can therefore have at most one exceptional
point, then, after possibly switching coordinates and passing to another higher
iterate of $f$, we can make two further assumptions:  
\begin{enumerate}
\item If $f$ has an exceptional point, then that point is the point at
  infinity.  
\item If $f$ does not have an exceptional point, then the point at
  infinity is fixed and unramified.  
\end{enumerate}

\begin{remark}\label{rm}
The conditions above ensure that if $(\infty, y)$ or $(x, \infty)$
appears in the intersection of $2d$ or more divisors $B_i$, then $\infty$
must be an exceptional point for $f$. 
\end{remark}

We now note that if two rational functions take on a large value at a
point, then that point must lie near a point in the intersection of
the two rational functions' pole divisors.  The following lemma is
quite similar to \cite[Lemma 2.1]{LevinRunge}.

\begin{lemma}\label{lem3}
Let $\phi_1, \dots, \phi_m$ be rational function on $\bP_1^2$.  Let $Z_1, \dots,
Z_m$ denote their pole divisors.  Suppose
that for $i \not= j$, we have that $Z_i$ and $Z_j$ do not share a
component; let $\Pi$ denote the (finite) set of points $p$ such that
$p \in Z_i \cap Z_j$ for some $i \not= j$.  Then for any $\delta >
0$ and any place $v$ of $K$, there is a computable constant $\gamma_{\delta,
  v}$ such that if $R \in \bP_1^2(K_v) \setminus \bigcup_{i=1}^m Z_i(K_v)$ and
$\Delta_v(R,p) \geq
\delta$ for all $p \in \Pi$, then
$$ | \phi_i(R) |_v \leq \gamma_{\delta, v}$$
for at least $(m-1)$ of $i \in \{1, \dots, m\}$.   
 \end{lemma}
\begin{proof}
  For each $i$ and any $\delta > 0$, one can effectively bound
  $|\phi_i|_v$ on the set of points that are at a distance of at least
  $\delta$ from the pole divisor $Z_i$.  To see this, note that it
  is clear when the poles are contained in fibers of $\bP_1^2$ and
  that otherwise there is an explicit embedding (a composition of
  Segre and $d$-uple embeddings) $\iota: \bP_1^2 \lra \bP_n$
  that sends $Z_i$ into a hyperplane $H$.
  Moreover, one can effectively compute a generator $h$ for the ideal sheaf
  of $H$ and a regular function $\psi$ on $\bP_n \setminus H$
  that restricts to $\phi_i$ on $\iota(\bP_1^2 \setminus Z_i)$, using the effective form of
  Hilbert's Nullstellensatz (see \cite{MW}).  One can then explicitly
  bound $\phi_i$ away from $Z_i$ in terms of the
  coefficients of $h$ and $\psi$.

Now, let $\cU$ denote the set of all points $R$ such that
$\Delta_v(R,p) \geq \delta$ for all $p \in \Pi$.  Then, for any
$\phi_i, \phi_j$ with $i \not= j$, there is a computable lower bound
$\epsilon_{i,j}$ on $\cU$ for the minimum of the distances of a point in
$\cU$ to $Z_i$ and $Z_j$.  Thus, there is some computable
$\gamma_{i,j}$ such that $\min ( | \phi_i(R) |_v , | \phi_j(R) |_v)
\leq \gamma_{i,j}$ for all $R \in \cU$.  Letting $\gamma_{\delta,
  v}$ be the maximum of these $\gamma_{i,j}$ gives the desired bound.
\end{proof}

In the case where $f$ has an exceptional point we need to treat this
point, which we have assumed is $(\infty, \infty)$, differently than
other points.  For an integer $i$, we let $d_i$ equal the degree of
$\beta_i(x,y)$, which is $d^{i} - d^{i-1}$ when $f$ has an exceptional
point.  We begin with a lemma about the behavior of the $\beta_i$ near
$(\infty, \infty)$ in the case where $\infty$ is an exceptional point
of $f$.  The idea here is simple: since any distinct $B_i$ and
$B_j$ meet transversally at $(\infty, \infty)$, at least one of
$|\beta_i(x,y)|_v$ and $|\beta_j(x,y)|_v$
can be bounded below in terms of its degree
near $(\infty, \infty)$.

\begin{lemma}\label{near_infty}
Let $p=(\infty, \infty)$ and suppose that $\infty$ is exceptional for
$f$.   Then, for any $M \not= N$, we have computable constants $C_{M,N,v}$ and a
computable $\delta > 0$ such that 
\begin{equation}\label{comp2}
\max \left( \frac{|\beta_M(x,y)|_v}{\max(|x|_v^{d_M},
    |y|_v^{d_M})},  
  \frac{|\beta_N(x,y)|_v}{\max(|x|_v^{d_N}, |y|_v^{d_N} )} \right) 
 \geq C_{M,N,v}
\end{equation} 
whenever $\Delta_v(P,p) < \delta$.
\end{lemma}
 
\begin{proof}
  Since $\infty$ is exceptional, $f$ is a polynomial.  Assume
  $(x,y)\in K_v^2$ with $|x|_v\leq |y|_v$. The largest degree homogeneous
  term in $\beta_N$ is a constant times
  $\frac{x^{d^N}-y^{d^N}}{x^{d^{N-1}}-y^{d^{N-1}}}$ and thus we have
\begin{equation}\label{big}
|\beta_N(x,y)|_v = |\epsilon y^{d_N}\tilde{\beta_n}(t)|_v+O(|y|_v^{d_N-1})
\end{equation}
for some constant $\epsilon$, where $t=x/y$ and $\tilde{\beta_n}$ is the monic
polynomial whose
roots are exactly the roots of unity of order dividing $d^N$ but not
$d^{N-1}$.   The $\tilde{\beta_n}$ are clearly pairwise coprime; thus, for
fixed $M \neq N$, we have
$$
\max(|\tilde{\beta_M}(t)|_v,|\tilde{\beta_N}(t)|_v)\geq W_{M,N,v} 
$$
for all $|t|_v \leq 1$ for some computable $W_{M,N.v}$.  Combining this with
\eqref{big} gives
$$\max \left( \frac{|\beta_M(x,y)|_v}{|y|_v^{d_M}},
  \frac{|\beta_N(x,y)|_v}{|y|_v^{d_N}} \right) \geq W_{M,N,v}  +
O(1/|y|_v)$$
for large $|y|_v$.  Similarly, when $|x|_v \geq |y|_v$, we obtain a bound
$$\max \left( \frac{|\beta_M(x,y)|_v}{|x|_v^{d_M}},
  \frac{|\beta_N(x,y)|_v}{|x|_v^{d_N}} \right) \geq W'_{M,N,v}  +
O(1/|x|_v)$$
for large $|x|_v$.  Combing these two bounds give \eqref{comp2} for
all $|x|_v, |y|_v > 1 / \delta$ for some computable $\delta > 0$.  When
$\Delta(x,y) < \delta$,
we have $|x|_v, |y|_v > 1/\delta$, so our proof is complete.  
\end{proof}

We now treat the case where $(\xi, \eta)$ is contained in the
intersection of at least $2d$ divisors $B_i$ but neither $\xi$ nor
$\eta$ is exceptional for $f$.  Again we take advantage of the fact
that the $B_i$ meet transversally at these $(\xi, \eta)$ to show that
for $i \not= j$, at least one of $|\beta_i(P)|_v$ and $|\beta_j(P)|_v$
can be bounded below in terms of the multiplicity of its pole divisor at
$(\xi, \eta)$ for points $P$ near $(\xi, \eta)$.

\begin{lemma}\label{lem1}
Let $p=(\xi,\eta) \not= (\infty,\infty)$
be contained in the intersection of $2d$ divisors $B_{m_0},\dots,
B_{m_{2d-1}}$.  Let $P=(\xi+t,\eta+u)$.  Then there is an increasing sequence $r_i$ with $\lim_{i \to \infty}
r_i/(d^i - d^{i-1}) = 0$ such that for any $M > N >
m_{2d -1}$, we have computable constants $C_{M,N,v}$ and a
computable $\delta > 0$ such that 
\begin{equation}\label{comp}
\max \left( \frac{|\beta_M(P)|_v}{\max(|t|_v^{r_M}|, |u|_v^{r_M})}),
  \frac{|\beta_N(P)|_v}{\max(|t|_v^{r_N}|, |u|_v^{r_N})} \right) \geq C_{M,N,v}
\end{equation} 
whenever $\Delta_v(P,p) < \delta$.
\end{lemma}

\begin{proof}
Since the point $p=(\xi,\eta)$ is contained in the intersection of $2d$ divisors
$B_{m_0},\dots, B_{m_{2d -1}}$, Lemma \ref{multiple} implies that $f^{\circ m_i
- 1}(\xi) = f^{\circ m_i - 1}(\eta)$ is a periodic, and hence fixed, critical
point for $f$ .  If that fixed critical point is the totally ramified point,
then $\xi=\eta=\infty$.  Since $p\neq (\infty, \infty)$, we may assume $f^{\circ
m_i-1}(\xi)=f^{\circ m_i-1}(\eta) = \theta$ is not a totally ramified point for
$f$ (see Remark~\ref{rm}); let $r$ be its ramification index.  Then, for any $Q\ge1$, we have
$$
f^{\circ Q}(\theta+t_1) = \theta + \gamma t_1^{r^Q}+\dots
$$
with $r<d$ and
$$
f^{\circ Q}(\theta+t_1)-f^{\circ Q}(\theta+u_1) = \gamma' (t_1^{r^Q}-u_1^{r^Q})
+\dots
$$
Write $h=m_{2d -1}$.  Taking $N=Q+h$ and denoting $f^{\circ h}(\xi+t) = \theta +
\lambda
t^{\sigma}+\dots$ and  $f^{\circ h}(\eta+u) = \theta + \lambda' u^{s}+\dots$ with $\lambda,\lambda'\neq 0$, 
we obtain
\begin{align*}
|f^{\circ N}(\xi+t)-f^{\circ N}(\eta+u)|_v &=  |f^{\circ Q}( f^{\circ h} (\xi+t)
) - f^{\circ Q}( f^{\circ h}( \eta+u))|_v \\
&= |\epsilon t^{\sigma r^Q}- \epsilon ' u^{sr^Q}|_v + o(1)
\end{align*}
for some constants $\epsilon, \epsilon'$.  Suppose that $|t^\sigma|_v \leq
|u^s|_v$. 
Then we have
\begin{equation}\label{u}
\begin{split}
|\beta_N(\xi + t, \eta + u)|_v &= |\epsilon''' |_v \left| \frac{t^{\sigma r^Q}-
u^{sr^Q}}{t^{\sigma r^{Q-1}}- u^{sr^{Q-1}}} \right|_v +o(1)\\
&=|u|_v^{s(r^Q - r^{Q-1})}(|\beta_N^*(\rho)|_v + o(1))
\end{split}
\end{equation}
where $\rho = t^{\sigma}/u^s$, and $\beta_N^*$ is polynomial
whose roots are the roots of unity of order dividing $r^Q$ but not $r^{Q-1}$. 
Therefore, the distinct $\beta_N^*$ are pairwise coprime, and so for
$M\not=N$, we have
\begin{equation}\label{coprime}
\max(|\beta_M^*(\rho)|_v,|\beta_N^*(\rho)|_v)\ge W_{M,N,v}>0
\end{equation}
for $|\rho|_v\le1$.  If $|t^\sigma|_v \geq |u^s|_v$, we divide instead by
$t^{\sigma}$ and obtain a bound like \eqref{u} for $|\beta_N(\xi + t, \eta +
u)|_v$ in $|t|_v^{\sigma(r^Q - r^{Q-1})}$, also in
terms of coprime polynomials for which a bound analogous to \eqref{coprime}
holds. \medskip Letting $r_N = \max(\sigma, s) r^Q$, and specifying that $|t|_v,
|u|_v < 1$ gives \eqref{comp}.   Finally, since $r<d$ and $Q=N-h$, we
have 
$$\lim_{N \to \infty}(r^Q - r^{Q-1}) / (d^N - d^{N-1}) = 0,$$
so $\lim_{N \to \infty}  r_N/(d^N - d^{N-1}) \rightarrow 0$.  
\end{proof}

Now, we introduce new rational functions $\phi_{N_i}$ on $\bP_1^2$
having poles where  $\beta_{N_i}$ have zeros, but where the behavior of
$|\psi_{N_i}(P)|_v$
can be controlled near certain points in the intersection of at least
$2d$ pole divisors $B_N$.  We will write 
$$ \Phi_N(x,y) = \frac{\alpha_N(x,y)}{\beta_N(x,y)}$$
where $\alpha_N$ vanishes to a high degree at these points and where
$\alpha_N$ and $\beta_N$ have the same pole divisors.

We begin by writing $\beta_N(x,y)$ as a quotient of degree $2 (d^N
- d^{N-1})$ bihomogeneous polynomials as 
\begin{equation}\label{hom}
\begin{split}
\frac{G_N(x_0,x_1; y_0, y_1)}{H_N(x_0,x_1; y_0, y_1)}.
\end{split}
\end{equation}
If $f$ is not a polynomial, then there are no points in the intersection of at
least $2d$ divisors $B_i$ at which any $H_N$ vanishes.  To see this, note that
the product of all $\beta_N$ up to $M$ is $$\frac{P_M(x_0,x_1)Q_M(y_0, y_1) -
Q_M (x_0,x_1)   P_M(y_0,y_1)}{Q_M(x_0, x_1) Q_M (y_0,y_1)}$$ where $P_M$ and
$Q_M$ are the homogeneous quotients of $f$, as in Section~\ref{nota}.   Thus, if
$H$ vanishes at $(a,b)$, then some iterate of $a$ or $b$ under $f$ is the point
at infinity.  On the other hand, if $(a,b)$ is in the intersection of at least
$2d$ divisors, then some iterate of $a$ and some iterate of $b$ is a periodic
ramification point.  But we have chosen coordinates so that no iterate of the
point at infinity is a periodic ramification point, proving the claim. 
We can rewrite \eqref{hom} as
\begin{equation}\label{h}
\begin{split}
\frac{g_N(x,y)}{h_N(x,y)}
\end{split}
\end{equation}
where $x = x_0/x_1$ and $y = y_0/y_1$ as before.  Then $\deg g_N \geq
d^N - d^{N-1}$, since $\deg g_N = d^N - d^{N-1}$ if $f$ is a polynomial and
$\deg g_N \geq 2(d^N - d^{N-1}) - 1$ otherwise, since $f$ fixes the point at
infinity.  We  see that if $f$ is a polynomial, then $h(x,y)$ is a constant (and
hence vanishes nowhere).  We let $d_N$ denote $\deg g_N$ (note that is is
consistent with the definition of $d_N$ before Lemma~\ref{near_infty}).

\begin{lemma}\label{lem2}
  Let $L>0$ be a positive integer, and $\Pi_{\infty}$ a finite set
  of points contained in at least $2d$ of the divisors $B_i$.  Then,
 there are rational functions $\phi_{M_1}, \dots, \phi_{M_L}$ on $\bP_1^2$, with $\phi_{M_i}$ having pole divisor
$B_{M_i}$, where $M_L > \cdots > M_1$,
  and $A_v>0$, $\delta > 0$ such
that whenever $\Delta_v(P,p)<\delta$ for some $p\in
\Pi_{\infty}$, we have $|\phi_{M_\ell}(P)|_v<A_v$ for all but at most one
$\ell \in \{1,\dots, L\}$.
\end{lemma}

\begin{proof}
Write $\Pi_{\infty}=\{p_1,\dots,p_k\}$. Take $N$ large enough so that the
conclusion of Lemma \ref{lem1} holds for each point $p_j=(\xi_j, \eta_j)\neq
(\infty, \infty)$ in $\Pi_\infty$, and write $r_{j,N}$ for the maximum of the
ramification indices of $f^{\circ N}$ at $\xi_j$ and $\eta_j$.  Let $e(j)$ be
the product of the degrees of $\xi_j$ and $\eta_j$ over $K$, and let
$g_{\xi_j}(x)$ and $g_{\eta_j}(y)$ denote their respective minimal polynomials
over $K$.  Since $r_{j,N}/d_N \longrightarrow 0$, increase $N$ so that
$\sum_{j=1}^k 2e(j) r_{j,N} < d_N$ by Lemma~\ref{lem1}. Then 
$$ \prod_{j=1}^k (g_{\xi_j}(x)g_{\eta_j}(y))^{r_{j,N}}$$ 
is a polynomial with coefficients in $K$ vanishing at each of the
points $p_j=(\xi_j, \eta_j)\neq (\infty, \infty)$ with multiplicity at least
$r_{j, N}$.  Furthermore, denoting $e_N$ for the degree of $\alpha_N$, we find
that $e_N \le \sum_{j=1}^k 2e(j) r_{j,N} < d_N$.  We let
\begin{equation}\label{alpha}
\alpha_N(x,y) = \frac{\prod_{i=1}^k
    (g_{\xi_i}(x)g_{\eta_i}(y))^{r_{j,N}} (xy)^{d_N - \sum_{j=1}^k
2e(j) r_{j,N} }}  {h_N(x,y)},
\end{equation}
where $h(x,y)$ is as in \eqref{h}.  Then the pole divisor of
$\alpha_N(x,y) / \beta_N(x,y)$ is simply the zero divisor of
$\beta_N(x,y)$.

Let $\phi_{M_1}, \dots, \phi_{M_L}$ be the rational function
$\alpha_{M_i}(x,y)/\beta_{M_i}(x,y)$, where $M_L > \dots > M_1$ are
sufficiently large as above.   Let $p_j \not= (\infty, \infty)$ be a
point in $\Pi_\infty$.  By \eqref{alpha}, we have a computable constant
$ E_{M_i, j, v}$ such that 
$$ |\alpha_{M_i} (\xi + t, \eta + u)|_v \leq E_{M_i, j, v} |t|_v^{r_{j,N}}
|u|_v^{r_{j,N}}$$  
for small $|t|_v, |u|_v$ by \eqref{alpha}.   Thus, by \eqref{comp},
we have for any $i \not=k$ that
$$ \min( |\phi_{M_i} (\xi + t, \eta + u)|_v,   |\phi_{M_k} (\xi + t,
\eta + u)|_v) \leq A_{M_i, M_k, j, v}$$
for some computable  $A_{M_i, M_k ,j, v}$.  Letting $A_{j,v}$ be
the maximum of these $A_{M_i, M_k, j, v}$ gives $|\phi_{M_\ell}(P)|_v<A_{j,v}$
for all but at most one $\ell \in \{1,\dots, L\}$ when $\Delta_v(P, p_j) <
\delta$, for some computable $\delta > 0$.  

Now, suppose that $p_j = (\infty, \infty)$.  Then $f$ is a polynomial
so $\beta_{N_i}$ and $\alpha_{N_i}$ are each  polynomials of degree
$d^{N_i} - d^{N_i -1}$.  Thus, we have 
$$\frac{|\alpha_{N_i}(x,y)|_v}{\max (|x|_v^{ d_{N_i}}, |y|_v^{ d_{N_i}} )} <
U_{j,v}$$
for some computable constant $U_{j,v}$.  Then, \eqref{comp2} implies that
for any $i \not= k$, there is a constant $A_{M_i, M_k, j,v}$ we have 
$$ \min(|\phi_{M_i}(x,y)|_v,  |\phi_{M_k}(x,y)|_v) \leq A_{M_i, M_k,
  j,v}$$ 
for all sufficiently large $|x|_v, |y|_v$.  Taking $A_{j,v}$ to be the maximum
of these  $A_{M_i, M_k, j,v}$ gives $|\phi_{M_\ell}(P)|_v<A_{j,v}$ for all but
at most one $\ell \in \{1,\dots, L\}$ when $\Delta_v(P, p_j) < \delta$.  
 
Finally, taking $A_v$ to be the maximum of all the $A_{j,v}$ gives the
desired $A_v$ and our proof is complete.
\end{proof}

Now, we use the Lemma~\ref{lem1} to produce a set of rational functions,
each having a pole divisor supported on some $B_i$, 
that are bounded away from a finite set of points contained in the
intersection of at least $2d$ $B_i$ at each place $v \in S$.  We use
induction on the size of $S$.

\begin{proposition}\label{s0}
There exists a set $\Psi$ of rational functions $\psi_{N_i}$ with
$i\in\{1,\dots, 2ds+1\}$, each having pole divisor with support on $B_{N_i}$,
and a finite set of points $\Pi = \{p_1, \dots, p_k \}$, each contained in at
least $2d$
distinct $B_{N_i}$  such that, for any $\delta > 0$, there are effectively
computable constants $T_\delta$ with the property that for any $Q$ that is
$S$-integral relative to $\sum B_{N_i}$, one of the following conditions holds:
\begin{enumerate}
\item there is a $\psi \in \Psi$ such that $|\Psi(Q)|_v \leq T_\delta$
  for all $v \in S$;
\item for each $v \in S$, there is a $p_v \in \Pi$ such that
  $\Delta_v(Q,p_v) < \delta$.  
\end{enumerate}
\end{proposition}
\begin{proof}
Let $s = |S|$.  We say that Property (*) holds for $s_0 \leq s$ if  there exists
a set $\Psi$ of rational functions $\psi_{N_i}$, each having pole divisor with
support on $B_{N_i}$, and a finite set of points $\Pi = \{p_1, \dots, p_k \}$,
each contained in at least $2d$ distinct $B_{N_i}$  such that, for any
$\delta > 0$, there are effectively computable proper constants $T_\delta$ with
the property that for any $Q$ that is $S$-integral relative to all of the
$B_{N_i}$, one of the following conditions holds:
\begin{enumerate}
\item there is a $\psi \in \Psi$ such that $|\psi(Q)|_v \leq T_\delta$ for all
$v \in S$;
\item there is a subset $S_0 \subseteq S$ with $|S_0| = s_0$ such that for
each $v \in S_0$, there is a $p_v \in \Pi$ such that $\Delta_v(Q,p_v) <
\delta$. 
\end{enumerate}

We will prove, by induction, that for every $s_0 \leq s$, Property (*) holds
for $s_0$.  If $s_0 = 0$, then (ii) holds since it is vacuously true for the
empty set.  Now suppose that Property (*) holds for some $s_0 \leq s - 1$, and
let $\Pi$ and $T_\delta$ be the set and constants for which (i) or (ii) holds.  
By Lemma \ref{lem2}, there is a computable proper $\delta_0 > 0$ such that we
may construct a set $\Phi$ of $2ds+1$ rational functions   $\phi_{M_1},
\dots, \phi_{2ds +1}$, with $\phi_{M_i}$ having pole   divisor $B_{M_i}$,
such that if $\Delta_v(R,p)<\delta_0$ for some $p_v \in \Pi$, then 
\begin{equation}\label{B}
\text{$|\phi_{M_i}(R)|_v < A_v$ for all   but at most one $i\in\{1,\dots,
2ds+1\}$.}
\end{equation}

Let $\Pi^*$ denote the finite set of points at which the $B_{M_i}$
intersect.  We let $\Pi_0$ denote those at which at most $2d -1$ meet.
Choose $\delta_1$ small enough such that if $p_v \in
\Pi_0$ is in $B_{M_{i_1}} \cap \dots \cap B_{M_{i_e}}$ for distinct
$\{i_1, \dots i_e \}$ and is in no other $B_{M_j}$ (note that $e \leq (2d-1)$),
then $\Delta_v(R,p_v) >\delta_1$ for all $R \in B_{M_k}$ for $k \notin \{i_1,
\dots, i_e \}$. Take any $R \in \bP_1(K)$ outside the support of the $B_{M_i}$. 
If $\Delta_v(R,p_v) < \delta_1$ for some $p_v \in \Pi_0$, where $p_v \in
B_{M_{i_1}} \cap \cdots \cap B_{M_{i_e}}$ for distinct $\{i_1, \dots
i_e \}$, then $|\phi_{M_i}(R)|_v$ is bounded by some computable
$\omega_v$ for all all but at most one $i \notin \{i_1, \dots i_e \}$ by our
choice of $\delta_1$, by Lemma~\ref{lem3}, so in particular $|\phi_{M_i}(R)|_v
\leq \omega_v$ for all but at most $2d$ of the $M_i$.

Now, let $\delta < \min(\delta_0, \delta_1)$.  Suppose further that
$\delta$ is less than half the minimum of the distances between
distinct points of $\Pi_0$.  If $\Delta_v(R,p_v) \geq \delta$ for all
$p_v \in \Pi^*$, then for all but at most one of
$\phi_{M_1}, \dots, \phi_{M_{2ds + 1}}$ we have $|\phi_{M_i}(R)|_v
\leq \gamma_v$ for some computable $\gamma_v$ by
Lemma~\ref{lem3}.  Putting this together, we see that since there are
$s$ places in $S$ and $2ds + 1$ functions $\phi_{M_i}$, this means that
there is some $\phi_{M_\ell}$ such that $|\phi_{M_\ell}(R)|_v$ is
bounded by $\max(\gamma_v, \omega_v, A_v)$ for all $v \in S$ unless
$\Delta_v(R, p_v^*) < \delta$ for some $v \in S $ and some $p_v^* \in
\Pi^* \setminus (\Pi_0 \cup \Pi)$.

Let $Q$ be a point that is $S$-integral relative all of the $B_{N_i}$
and all of the $B_{M_i}$. Suppose that there is no $\psi_{N_i}$ such
that $|\psi_{N_i}(Q)|_v$ is bounded by $T_\delta$, where ($T_\delta$ is
as in the statement of Property (*)) for all $v \in S$.  Then there is a subset
$S_0 \subset S$ with $|S_0| = s_0$ such that for each $v \in S_0$, there is
a $p_v \in \Pi$ such that $\Delta_v(Q,p_v) < \delta$, by the inductive
hypothesis.  Suppose that there is no $|\phi_{M_\ell}(R)|_v$ that is
bounded by $\max(\gamma_v, \omega_v, A_v)$ for all $v \in S$.  Then, as
above, there is a $v \in S$ such that $\Delta_v(Q,p_v^*) < \delta$ for
some $p_v^* \in \Pi^* \setminus (\Pi \cup \Pi_0)$; no two elements of
$\Pi^*$ are within $2 \delta$ of each other, we must have $v \notin
S_0$.  Thus there is a set $S'$ of size $s_0 + 1$ such that for each
$v \in S'$, there is $p_v \in \Pi \cup (\Pi^* \setminus \Pi_0)$ such
that $\Delta_v(Q,p_v) < \delta$.  Hence, letting $\Psi' = \Psi \cup
\Phi$ and $\Pi' = \Pi \cup (\Pi^* \setminus \Pi_0)$, we see that
Property (*) holds for $s_0 + 1$.  This completes the inductive step and
our proof is done.\end{proof} 

Finally, we use Proposition~\ref{s0} to prove Theorem~\ref{eff-degen}.  

\begin{proof}[Proof of Theorem~\ref{eff-degen}]
Let $\Psi$ and $\Pi$ be the sets given by Proposition~\ref{s0}.  Recall that
each $\psi_{N_i} \in \Psi$ has a pole divisor with support on $B_{N_i}$.   
Applying Lemma~\ref{lem2}, we obtain $s + 1$ rational functions $\phi_{M_i}$
each having pole divisor with support on $B_{M_i}$ along with computable $\delta
> 0$ and constants $\gamma_v$ such that whenever $\Delta_v(P,p_v)<\delta$ for
some $p_v \in \Pi$, we have $|\phi_{N_i}(P)|_v<\gamma_v$ for all but at most one
$i \in \{1,\dots, s+1\}$.  Take any $Q$ that is $S$-integral for $D_{k_0}$ where
$k_0$ is the maximum of all the $N_i$ and $M_i$ above.   Then we have a constant
$\kappa$ such that  $\sum_{v \notin S} \max(\log |\psi_{N_i}(Q)|_v,0) \leq
\kappa$ and $\sum_{v \notin S}\max( \log |\phi_{M_i}(Q)|_v,0) \leq \kappa$
for all $v \in S$.  Furthermore, by Proposition~\ref{s0} we have one of the
following:
\begin{enumerate}
\item there is a $\psi_{N_i} \in \Psi$ such that $|\psi_{N_i}(Q)|_v \leq
T_\delta$ for all $v \in S$; 
\item for each $v \in S$, there is a $p_v \in \Pi$ such that $\Delta_v(Q,p_v)
< \delta$.  
\end{enumerate}
If (i) holds, then there is some there is some $\psi_{N_i}$ such that
$$\sum_{\text{places $v$ of $K$}}\ \max(|\log \psi_{N_i}(Q)|_v,0) \leq \kappa +
\sum_{v \in S} T_\delta.$$

If (ii) holds, then for each   $|\phi_{M_i}(P)|_v<\gamma_v$ for all but at most
one $i \in \{1,\dots, s+1\}$.  Since there are only $s$ places in $S$, this
means there is some $\phi_{M_i}$ such that $|\phi_{M_i}(P)|_v<\gamma_v$  for
all $v \in S$.  Hence, there is some $\phi_{M_i}$ such that $$\sum_{\text{places
$v$ of $K$}} \max( \log |\phi_{M_i}(Q)|_v, 0) \leq \kappa + \sum_{v \in S}
\gamma_v.$$

Since there are only finitely many points $z \in K$ of bounded height and
$Q$ lies on a curve of the form $\psi_{N_i}(x,y) = z$  or $\phi_{M_i}(x,y) = z$,
with $z \in K$ having bounded height, we see then that $Q$ lies on an
effectively computable proper subvariety of $\bP_1^2$, as desired.  
\end{proof}

\section{Effective finiteness}\label{fini}

Silverman mentions that
\cite[Theorem A]{SilSiegel} can be made effective.  For the sake of
completeness, we give a quick proof of this fact.

\begin{theorem}\label{silverman-eff}
Let $K$ be a number field, let $S$ a finite set of primes in $K$, let $f: \bP_1
\lra \bP_1$ be a rational function with degree $d \geq  2$, let $a$ be a point
that is not periodic for $f$, and let $b$ be a point that is not exceptional for
$f$.   Then the set of $n$ such that $f^{\circ n}(a)$ is integral relative to
$b$ is finite and effectively computable.  
\end{theorem}

\begin{proof}
  Since $b$ is not exceptional, $f^{-4}(b)$ contains at least three
  distinct points. To see this note that $f^{-2}(b)$ contains at
  least two points, since $b$ is not exceptional. If $f^{-2}(b)$
  contains exactly two points, then there is a totally ramified point
  in $f^{-1}(b)$ or $f^{-2}(b)$.  This point cannot be fixed by $f$ so it
  cannot be in both $f^{-3}(b)$ and $f^{-4}(b)$.  If $f^{-3}(b)$
  contains only two points, then they must both be totally ramified,
  so $f^{-4}(b)$ must contain a point that is not totally ramified
  (because $f$ has at most two totally ramified points, by
  Riemann-Hurwitz), which means that  $f^{-4}(b)$ contains at least
  three points.  
  
For $n \geq 3$, we have that $f^{\circ n}(a)$ is
  $S$-integral relative to $b$ if and only if $f^{\circ (n-3)}(a)$ is
  $S$-integral relative to the points in $f^{-2}(b)$.  Changing
  coordinates, these $f^{\circ (n-3)}(a)$ are solutions to the $S$-unit
  equation, which has an effective solution (see \cite[Theorem
  5.4.1]{BG}, for example).
\end{proof}

Now we can prove Theorem~\ref{main}.

\begin{proof}[Proof of Theorem \ref{main}]
  Theorem \ref{eff-degen} delivers an effectively computable
  one-dimensional subvariety $Z$ such that the $(m, n)$ with $m, n
  \geq k_0$ for which $f^{\circ m}(u)$ is $S$-integral relative to
  $f^{ \circ n}(w)$ are effectively computable for all $(f^{\circ
    m}(u), f^{\circ n}(w))$ outside of $Z$.

Let $c$ be the number of components of $Z$.  Let $I_{u,w}$ denote the
set of $(m,n)$ such that $f^{\circ m}(u)$ is $S$-integral relative to
$f^{\circ n}(w)$.   By
Theorem~\ref{silverman-eff}, we know set of $(m,n) \in I_{u,
  w}$ with $\min(m,n) \leq c + k_0$ is effective computable.
Thus, it suffices to show that the set of $(m,n) \in I_{u,w}$
with $\min(m,n) \geq c + k_0$ and $(f^{\circ m}(u), f^{\circ 
  n}(w)) \in Z$ is effective computable.  Note that if $m , n \geq
c \geq r$, then $f^{\circ m}(u)$ can be $S$-integral relative
to$f^{\circ n}(w)$ only when $f^{\circ (m-r)}(u)$ is $S$-integral
relative to $f^{\circ (n-r)}(w)$.  Hence, it suffices to find all $(m,n)
\in I_{u,w}$ such that $(f^{\circ m}(u), f^{\circ 
  n}(w))$ is in $$ Z \cap (f,f)^{-1}(Z) \cap \dots \cap
(f,f)^{-c} (Z).$$ If this is finite, we are done.  Otherwise, there
is a common component $X$ among $Z, (f,f)^{-1}(Z), \dots,
(f,f)^{-c} (Z)$. Then $(f,f)^i(X)$ is a component of $Z$ for $i =
0, \cdots, c$. Therefore, $(f,f)^{ i}(X) = (f,f)^{ j}(X)$ for
some $c \geq j > i \geq 0$.  So $(f,f)^i(X)$ is a periodic component
of $Z$.

Thus, we are left to show that the points of the form $(f^{\circ m}(u),
f^{\circ n}(w))$ which $S$-integral relative to $D_0$ on any periodic curve $X$ for
$(f,f)$ can be computed.  Now, since $X$ admits a self-map of positive
degree, it must have genus 0 or 1.  Since $X \cap D_0$ contains at
least one point, we see that if $X$ has genus 1, then the integral
points on $X$ relative to $D_0$ can be effectively computed (see
\cite{BC, Bilu-Siegel}).  If $X$ has genus 0, and $X \cap D_0$
contains a nonexceptional point, then we are done by
Theorem~\ref{silverman-eff}.  If $X \cap D_0$ contains only an
exceptional point $z$, then after changing coordinates, we may write
the restriction of $(f,f)^2$ to $X$ as a polynomial $P(t)$, where $z$
is the point at infinity. If $X \cap D_0$ contains two exceptional
points for $P(t)$, the, after changing coordinates, , we
may write the restriction of $(f,f)^2$ to $X$ as a polynomial $t^n$, where $z_1$
is zero and $z_2$ is the point at infinity.  In either case, after expanding $S$
to a possibly larger set of primes $S'$ we have that for any $S'$-integral point
$\gamma$ on $X$, each iterate $(f,f)^{\circ 2i}(\gamma)$ is $S'$-integral
relative to $D_0$.  This means that there are infinitely many $S'$-integral
points relative to $D_0$ on $X$, which contradicts the the main theorem of the
Appendix of \cite{BGKTZ}.
\end{proof}

\section{Cyclic and exceptional cases}

When $f$ is conjugate to a powering map, we do not obtain a finiteness
result.  This can be seen, for example, by considering the map $f(x) =
x^3$ and the points $u = 2$, $w = -2$.  Then, if $S$ is the set
containing the archimedean place and the place 2, we have that
$f^{\circ m}(u)$ is $S$-integral relative to $f^{\circ m}(w)$ for all
$m$.  On the other hand, it is possible to give a reasonable description of the
$(m,n)$ such
that  $f^{\circ m}(u)$ is $S$-integral relative to $f^{\circ n}(w)$.  

In \cite{GTZ2}, it is proved that if $\deg a, \deg b > 2$ for
polynomials $a$ and $b$, then the set of $(m,n)$ such that
$a^{\circ m}(u) = b^{\circ n}(v)$ forms a finite union of cosets of
subsemigroups of $\bN^2$ (that is a finite union of additive
translates of subsets of $\bN^2$ that are closed under
addition). Here, $\bN$ is considered to include 0 so any finite set of
$(m,n)$ is a finite union of cosets of $(0,0)$.   

For a set of places $S'$ containing all the archimedean places, we define
$$ I_{u,w, S'} = \{ (m,n) \in \bN^2 \; \mid \; \text{$f^{\circ m}(u)$ is
  $S'$-integral relative to $f^{\circ n}(w)$} \}.$$

\begin{proposition}\label{cyclic}
  Let $f$ be conjugate to $x^{\pm d}$, let $S$ be a finite set of
  places of $K$.  Then, for some finite set of places $S'$ with $S
  \subseteq S'$, the set $I_{u, w, S'}$ a finite union of effectively
  computable cosets of subsemigroups of $\bN^2$.  Furthermore, the set
  $I_{u, w, S'}$ is finite if $u$ and $w$ are multiplicatively
  independent or if $u$ and $w$ are in the cyclic group generated by a
  non-torsion element of $K^*$.
\end{proposition}
\begin{proof}
  After changing coordinates by an automorphism $\sigma \in
  \PGL_2(K')$, for $K'$ a finite extension of $K$, we can write
  $\sigma f \sigma^{-1}(x)
  = x^{\pm d}$.  Choose a set $S'$ of primes that includes both all
  the primes appearing in the coefficients or determinant of $\sigma$
  as well as all the primes lying over primes in $S$;  then for any
  $P,Q \in \bP_1(K')$, we have that $P$ is $S'$-integral relative to
  $Q$ if and only if $\sigma P$ is $S'$-integral relative to
  $\sigma Q$ (note that this choice of $S'$ depends only on $f$, not
  on $u$ or $w$).  Thus, it suffices to prove the theorem when $f(x) =
  x^{\pm d}$. 

  \smallskip If $f(x) = x^{-d}$, then by considering the orbit of $(u,
  w)$ along with those of $(f(u), w)$, $(u, f(w))$, and $(f(u),
  f(w))$, we reduce to the case where $f(x) = x^{d^2}$ for some
  $d$. If $u$ or $w$ is zero or infinity, the conclusion is obvious.
  If neither $u$ nor $v$ is 0 or infinity, we may assume that $u$ and
  $w$ are both $S'$-units after expanding $S'$.  Then $f^{\circ m}(u)
  - f^{\circ n}(w)$ is an $S'$-unit if and only if $\frac{f^{\circ
      m}(u)}{f^{\circ n}(w)} - 1$ is an $S'$-unit.  Thus, if
  $(f^{\circ m}(u), f^{\circ n}(w))$ is $S'$-integral relative to
  $D_0$, then it lies on a curve of the form $x - y = \tau y$ where
  $\tau$ is an $S'$-unit such that $\tau + 1$ is also an $S'$-unit.
  By \cite[Theorem 5.4.1]{BG}, the set of such $\tau$ is finite and
  effectively computable.  Thus, if $u$ and $w$ are multiplicatively
  independent, then $\frac{f^{\circ m}(u)}{f^{\circ n}(w)}$ takes on
  any such value $\tau$ at most once, so there are at most finitely
  many $(m,n)$ such that $f^{\circ m}(u)$ is $S'$-integral relative to
  $f^{\circ n}(w)$.  For any fixed value of $1+\tau$, the set of $m,n$
  such that $u^{d^{2m}} / w^{d^{2n}} = 1 + \tau$ clearly forms a
  finite union of cosets of subsemigroups of $\bN^2$.

  If $u$ and $w$ are both in the subgroup of $K^*$ generated by a
  single element $z$ that is not a root of unity, then we may write $u
  = z^A$, $w = z^B$.  Then, we have $z^{A d^m - B d^n} = (1 + \tau)$
  for one of the finitely many $1 + \tau$ above whenever $f^{\circ m}(u)$ is
  $S'$-integral relative to $f^{\circ n}(w)$.  Now, for any constant $C$, the
  set of $(m,n)$ such that $A d^m - B d^n = C$ is finite
   unless $C = 0$ (since $\gcd(A d^m, B d^n) \to \infty$ to
  infinity as $\min(m, n) \to \infty$), but when $C =0$,we have $\tau = 0$, which is
  not an $S$-unit.  Hence, in this case there are at most finitely many $(m,n)$ such that
  $f^{\circ m}(u)$ is $S'$-integral relative to $f^{\circ n}(w)$.
\end{proof}
 
When at least one of $u$ or $w$ is preperiodic, but neither of $u$ or
$w$ is exceptional, it is easy to see from Theorem~\ref{main} that the
set of $(m,n) \in \bN^2$ such that $f^{\circ m}(u)$ is $S$-integral
relative to $f^{\circ n}(w)$ forms a finite union of effectively
computable cosets of subsemigroups of $\bN^2$.  When $u$ or $w$ is
exceptional, however, one should not expect there to be a particularly nice
pattern to the set of $(m,n)$ such that $f^{\circ m}(u)$ is
$S$-integral relative to $f^{\circ n}(w)$.  Benedetto-Briend-Perdry
\cite{BBP} show that if $f(x) = x^2 + \frac{x}{p}$, and $v$ is the
point at infinity, then for any set $\cU$ of positive integers, there
is a point $u \in \bQ_p$ such that $f^{\circ m}(u) \in \bZ_p$ if and
only if $m \in \cU$; although this is only stated over $\bQ_p$, it is
very likely that one can find examples for many complicated infinite
$\cU$ over $\bQ$.  This problem can be overcome by enlarging $S$ to a
finite set of primes $S'$ including all the primes of bad reduction
for $S$.

\begin{proposition}\label{exceptional}
Suppose that $w$ is exceptional and that there is no $m$ such that
$f^{\circ }(u) = w$.  Then for some finite set of places $S'$ with
$S \subseteq S'$,  the set $I_{u, w,S'}$ is all of $\bN^2$.
\end{proposition}
\begin{proof}
  Arguing as in Proposition~\ref{cyclic}, we may change coordinates so
  that $f^{\circ 2}$ is a polynomial and $w$ is the point at infinity
  and enlarge $S$ to some $S'$ where our notion of $S'$-integrality is
  not affected by the coordinate change.  If we enlarge $S'$ further
  to include all of the places at which $u$, $f(u)$, or a coefficient
  of $f^{\circ 2}$ has a pole, then $f^{\circ 2m}(u)$ and $f^{\circ
    2m}(f(u))$ are $S'$-integral relative to $w$ for all $m$, so
  $I_{u,w,S'}$ is all of $\bN^2$.
\end{proof}
 
\section{Further questions}
 
If $f$ and $g$ are two rational functions of degree $d>1$  such that
there are no $z_1, z_2$ such that $f^{\circ 2}(z_1) =
g^{\circ 2}(z_2)$ with $f^{\circ 2}$ ramifying at $z_1$ and $g^{\circ
  2}$ ramifying at $z_2$ (a reasonably ``generic'' condition), then
$f^{\circ 2}(x) - g^{\circ 2}(y) = 0$ gives a nonsingular
curve corresponding to a divisor $D_2$ of type $(2d, 2d)$ on $\bP_1^2$.
Since $d \geq 2$, we have that $D_2 + K_X$ is ample for $K_X$ a
canonical divisor of $\bP_1^2$.  Thus, Vojta's conjecture
\cite[Conjecture 3.4.3]{PV}
would imply
that the set of $S$-integral points relative to $D_2$ must be
degenerate.  Hence, we may expect that an analog of 
Theorem~\ref{ineff-degen}  holds in this case.  It may be that the method of this paper allows one to prove such general statements.



\newcommand{\etalchar}[1]{$^{#1}$}
\def\cprime{$'$} \def\cprime{$'$} \def\cprime{$'$} \def\cprime{$'$}
\providecommand{\bysame}{\leavevmode\hbox to3em{\hrulefill}\thinspace}
\providecommand{\MR}{\relax\ifhmode\unskip\space\fi MR }
\providecommand{\MRhref}[2]{%
  \href{http://www.ams.org/mathscinet-getitem?mr=#1}{#2}
}
\providecommand{\href}[2]{#2}

\end{document}